\documentclass[11pt]{amsart}
\usepackage{amssymb,latexsym, amscd, a4wide}
\usepackage[all]{xy}
\usepackage{graphicx}
\usepackage{mathrsfs}
\usepackage{amsmath}
\usepackage{pb-diagram}

 \newtheorem{theorem}{Theorem}[section]
 \newtheorem{cor}[theorem]{Corollary}

 \newtheorem{lemma}[theorem]{Lemma}
 \newtheorem{proposition}[theorem]{Proposition} \theoremstyle{definition}
 \newtheorem{definition}[theorem]{Definition}
 \theoremstyle{definition}
 \newtheorem{example}[theorem]{Example}
 \theoremstyle{remark}
 
 \numberwithin{equation}{section}

\numberwithin{equation}{section}
\newcommand{\Ind}{\operatorname{Ind}}

\newcommand{\Hom}{\operatorname{Hom}}

\newcommand{\End}{\operatorname{End}}

\newcommand{\Coker}{\operatorname{Coker}}

\newcommand{\soc}{\operatorname{soc}}

\newcommand{\Proj}{\ensuremath{\operatorname{Proj}}}
\newcommand{\Inj}{\ensuremath{\operatorname{Inj}}}

\newcommand{\w}{\ensuremath{\widetilde}}

\newcommand{\F}{\ensuremath{\mathbf{k}}}

\newcommand{\Z}{\ensuremath{\mathbb{Z}}}

\newcommand{\Rmod}{\ensuremath{R\text{-mod}}}
\newcommand{\kmod}{\ensuremath{\mathbf{k}\text{-mod}}}
\newcommand{\kMod}{\ensuremath{\mathbf{k}\text{-Mod}}}
\newcommand{\ModF}{\ensuremath{\text{Mod-}F}}
\newcommand{\modF}{\ensuremath{\text{mod-}F}}
\newcommand{\RMod}{\ensuremath{R\text{-Mod}}}
\newcommand{\C}{\ensuremath{\mathbb{C}}}
\newcommand{\Cb}{\ensuremath{\operatorname{C}^b}}
\newcommand{\Kb}{\ensuremath{\operatorname{K}^b}}

\newcommand{\DA}{\ensuremath{\operatorname{D}^b(\A)}}
\newcommand{\KR}{\ensuremath{\operatorname{K}^b}(\Rmod)}

\newcommand{\Ker}{\operatorname{Ker}}
\newcommand{\Jm}{\operatorname{Im}}

\newcommand{\Id}{\operatorname{Id}}

\newcommand{\FF}{\operatorname{\mathbb{F}}}

\newcommand{\g}{\ensuremath{\mathfrak{g}}}

\newcommand{\Cab}{\ensuremath{\mathcal{C}}}
\newcommand{\KC}{\ensuremath{\operatorname{K}^b(\mathcal{C})}}

\newcommand{\DC}{\ensuremath{\operatorname{D}^b(\mathcal{C})}}

\newcommand{\Ss}{\ensuremath{\mathbb{S}}}
\newcommand{\V}{\ensuremath{\mathbb{V}}}
\newcommand{\VA}{\ensuremath{\mathbb{V|_{\A}}}}

\newcommand{\BGG}{\ensuremath{\mathcal{O}}}

\newcommand{\Loc}{\ensuremath{\operatorname{L}}}

\newcommand{\A}{\ensuremath{\mathcal{A}}}
\newcommand{\m}{\ensuremath{\mathfrak{m}}}
\newcommand{\Ob}{\ensuremath{\operatorname{Ob}}}

\newcommand{\pa}{\ensuremath{\partial}}

\newcommand{\Pp}{\ensuremath{P}}
\newcommand{\HHom}{\ensuremath{\mathcal{H}om}}

\begin{document}
\author{Erik Backelin and Omar Jaramillo}
\title[Auslander-Reiten sequences]{Auslander-Reiten sequences and $t$-structures
on the homotopy category of an abelian category.}
\address{Departamento de Matem\'{a}ticas, Universidad de los Andes,
Carrera 1 N. 18A - 10, Bogot\'a, COLOMBIA}
\email{erbackel@uniandes.edu.co}
\address{Departamento de Matem\'{a}ticas, Universidad de los Andes,
Carrera 1 N. 18A - 10, Bogot\'a, COLOMBIA}
 \email{od.jaramillo114@uniandes.edu.co}

\maketitle
\begin{abstract} We use $t$-structures on the homotopy category $\KR$ for an artin algebra $R$ and
Watts' representability theorem to give an existence proof for
Auslander-Reiten sequences of $R$-modules. This framework
naturally leads to a notion of generalized (or higher)
Auslander-Reiten sequences.
\end{abstract}

\section{Introduction}
Krause, \cite{K}, \cite{K2}, used Brown's representability theorem
for triangulated categories to give a short proof for the
existence of Auslander-Reiten (abbreviated AR) triangles. The
present article is a variation of that theme. We use
$t$-structures to define an abelian category $\A$ where
AR-sequences naturally occur as simple objects. We apply Watts'
representability theorem to $\A$ to reprove the existence of
AR-sequences for modules over an artin algebra $R$.

To set up the general constructions $\Cab$ will denote an abelian
category;  later on $\Cab$ will be $\Rmod$.

Let $\Cb(\Cab)$ be the category of bounded cochain complexes in
$\Cab$ and let $\KC$ be the category of cochain complexes with
morphisms modulo homotopy. We consider a $t$-structure $(D^{\leq
0}, D^{\geq 0})$ on $\KC$ which is \emph{standard} in the sense
that the localization functor maps it to the tautological
$t$-structure on the bounded derived category $\DC$, see
Proposition \ref{p1}. (We briefly investigate other standard
$t$-structures on $\KC$ as well.)

Let $\A = \A(\Cab) = D^{\leq 0} \cap D^{\geq 0}$ be the heart of
the $t$-structure. This is an abelian category whose objects are
complexes
$$[A \overset{f}{\to} B \overset{g}{\to} C]$$
such that $f$ is injective and $\Ker g = \Jm f$.

We observe that the category $\A$ is naturally equivalent to a
subcategory of the category $\mathfrak{Ab}^{\Cab^{op}}$ of
functors from $\Cab$ to the category of abelian groups. Moreover,
in this way the functor $C \mapsto P_C := [0 \to 0 \to C]$
corresponds to the Yoneda embedding $\Cab \to
\mathfrak{Ab}^{\Cab^{op}}$. We prove that $\DA = \KC$ and describe
injective objects of $\A$.

In the case when $\Cab$ is a finite length category we shall see
that simple objects of $\A$ - if they exist - are given by
Auslander and Reiten's almost split right maps, see Section
\ref{simple object section}.

The functor category $\mathfrak{Ab}^{\Cab^{op}}$ has since
Auslander been a central tool in AR-theory. Our $\A$ provides
merely a different realization of it, but a realization that we
prefer because it is more intuitive and from its definition it is
clear that it lives naturally inside the triangulated category
$\KC$. For instance, Auslander's defect of a short exact sequence,
living in $\mathfrak{Ab}^{\Cab^{op}}$, corresponds to the homotopy
class of the sequence in $\A$.

Now we specialize to the case $\Cab = \Rmod$. The main ingredient
in Krause's existence proof for AR-triangles in a triangulated
category was a Serre duality functor and the existence of this he
deduced from Brown's representability theorem following ideas of
Neeman, \cite{N}, \cite{N2}.

With this in mind it clear that in our setup the existence of
AR-sequences would follow from a Serre duality functor on $\KR$,
because one could use the $t$-structure to truncate it down to an
Auslander-Reiten type duality between short exact sequences (see
corollary \ref{AL corr in good language} for the precise meaning
of such a duality).

However, we cannot deduce the existence of this Serre duality from
Brown's representability theorem, because it is not known to us
whether $\operatorname{K}(\RMod)$ is well-generated (compare with
\cite{H-J}). Instead, our approach is the (far more elementary)
theory of abelian categories.

We use Watts' theorem to prove that any projective object $P_C$ in
the abelian category $\A$ has a ``Serre dual" object $S P_C \in
\A$; this implies a Serre duality on $\KR$, see Proposition
\ref{Serre dual on Kb}, and also AR-duality. From the latter we
deduce, with a proof similar to the one given in \cite{ARS}, the
existence of AR-sequences, Theorem \ref{theorem main}. We also
interpret the AR-sequence with end term $C$ as $\Jm \tau$, where
$\tau: P_C \to S P_C$ is a certain minimal map.

In fact, in order to prove these existence and duality theorems
for AR-sequences we could have altogether avoided to mention
triangulated categories and worked, ad hoc, within the abelian
category $\A$. However, the viewpoint of $t$-structures on $\KC$
is very valuable. It seems to be a natural source of the theory
and it allows us to rediscover or reinterpret familiar notions in
AR-theory, like the dual of the transpose, the defect and
projectivization. It also naturally generalizes:

\medskip

\noindent There exists a notion of higher AL-sequences (see
\cite{I}). In Section \ref{Generalized AR-sequences} we propose
another method to generalize AR-theory. We  define a
\emph{generalized AR-sequence} to be a simple objects in the heart
of a certain $t$-structure on $\KC$, where $\Cab$ is an additive
(not necessarily abelian) category. It is certain that these
generalized AR-sequences and the higher AR-sequences of \cite{I}
are intimately relate (perhaps they coincide), but we haven't
worked this out.

There are interesting examples of generalized AR-sequences. For
instance, Soergel's theory of coinvariants, \cite{S}, shows that a
block in the BGG category $\BGG$ of representations of a complex
semi-simple Lie algebra $\g$ is equivalent to a category of
generalized AR-sequences in $\KC$, with $\Cab$ is a certain
category of modules over the cohomology ring of the flag manifold
of $\g$. See Section \ref{Soergel section} for a discussion.

\smallskip

\noindent We approach some more themes: We discuss the rather
obvious fact why $\A$ fails to be a noetherian category and give a
brief discussion of duality on $\A$ in the case when $R$ is a
Frobenius algebra.
\subsection{Acknowledgements} We thank Chaitanya Guttikar and the referee for useful comments.

\section{Standard $t$-structures on $\KC$}
\subsection{Notations} We
denote by $\Ind(\Lambda)$ the class of iso-classes of
indecomposable objects in a category $\Lambda$. $\Proj(\Lambda)$
and $\Inj(\Lambda)$ are the full subcategories of $\Lambda$ whose
objects are projective and injective, respectively. Let $R$ be a
ring; $\Rmod$, $\text{mod-}R$, $\RMod$ and $\text{Mod-}R$, denote
the categories of finitely generated left, finitely generated
right, all left and all right $R$-modules, respectively.

\subsection{A standard $t$-structure on $\KC$ and its heart}
Let $\Cab$ be an abelian category. In this section we investigate
a specific standard $t$-structure on $\KC$ and its heart. We
briefly discuss two other specific standard $t$-structures and the
existence of more. We also give an example from representation
theory where standard $t$-structures naturally occur.

We follow the notations of \cite{KS} concerning $t$-structures and
triangulated categories. See also \cite{GM} and \cite{N2} for more
details. We define a $t$-structure $(D^{\leq 0}, D^{\geq 0})$ on
$\KC$ as follows. Put
\begin{equation}\label{leq0}
D^{\leq 0} = \{X \in \KC;\,X^i = 0 \hbox{ for } i > 0\}
\end{equation}
Here $X = \{X^i,d^i\}$. (To be more precise, the objects of
$D^{\leq 0}$ are complexes homotopic to the right hand side of
\ref{leq0}, but we omit this kind of linguistic precision.) Put
\begin{equation}\label{geq0}
D^{\geq 0} = \{X \in \KC; \, X^i = 0 \hbox{ for } i < -2, \;
H^{-2}(X) = H^{-1}(X) = 0\}
\end{equation}

\begin{proposition}\label{p1} $(D^{\leq 0}, D^{\geq 0})$
is a bounded standard $t$-structures on $\KC$.
\end{proposition}
\begin{proof}
\noindent \textbf{a)} For a morphism $f: (X,d_X) \to (Y, d_Y)$ in
$\KC$ the mapping cone $M(f)$ is defined by
\begin{equation}\label{Mf}
M(f)^n = X^{n+1} \oplus Y^n
\end{equation}
with differential given by $d_{M(f)}(x_{n+1},y_n) = (-d_Xx_{n+1}+
f(y_n), d_Y y_n)$. Note that the inclusion $D^{\leq 0}
\hookrightarrow \KC$ has as right adjoint the truncation functor
$\tau^{\leq 0}$ which is defined by $\tau^{\leq 0}(X) = $
\begin{equation}\label{form1}
\ldots \to X^{-n} \overset{d^{-n}}{\to}  \ldots \to X^{-2}
\overset{d^{-2}}{\to} X^{-1} \overset{d^{-1}}{\to} \Ker d^0 \to 0
\end{equation}
For $X \in \KC$, let $\alpha_X: \tau^{< 0}(X) \to X$ be the
natural map.

\smallskip

\noindent \textbf{b)} We show that $D^{\geq 0} = \{M(\alpha_X); X
\in \KC\}$. For $X \in \KC$ we have $M := M(\alpha_X) =$
\begin{equation}\label{c1}
\to X^{-n+1} \oplus X^{-n} \to \ldots \to X^{-2}\oplus X^{-3} \to
\Ker d^{-1} \oplus X^{-2} \to
 X^{-1} \to X^0 \to \ldots
\end{equation}
Hence $M \cong M' \oplus M''$ where $M'$ is the $0$-homotopic
complex
\begin{equation}\label{c2}
\ldots \to X^{-n+1} \oplus X^{-n} \to \ldots \to X^{-2}\oplus
X^{-3} \to X^{-2} \to 0
\end{equation}
where all differential are the same as in \ref{c1} except $d:
X^{-2}\oplus X^{-3} \to X^{-2}$ which is given by
$d(x_{-2},x_{-3}) = x_{-2} -dx_{-3}$. $M''$ is the subcomplex
\begin{equation}\label{c3}
0 \to \Ker d^{-1} \to X^{-1} \to X^0 \to \ldots
\end{equation}
of $M$. Thus $M \cong M''$ in $\KC$ which proves the statement of
\textbf{b)}.

\smallskip

\noindent \textbf{c)} It follows from \textbf{b)} that any $X \in
\KC$ fits into a distinguished triangle $\tau^{<0}X \to X \to M
\overset{+1}{\to}$ where $\tau^{<0}X \in D^{< 0}$ and $M \in
D^{\geq 0}$. It follows from the definitions that $\Hom_{\KC}(X,Y)
= 0$ for $X \in D^{\leq 0}$ and $Y \in D^{\geq 1}$. Thus $(D^{\leq
0}, D^{\geq 0})$ is a $t$-structure; clearly it is standard. Note
also that any object of $\KC$ belongs to $D^{\leq a} \cap D^{\geq
b}$ for some $a,b \in \Z$ which means that the $t$-structure is
bounded.
\end{proof}

It follows from Proposition \ref{p1} that the inclusion functor
$D^{\geq 0} \hookrightarrow \KC$ has the left adjoint $\tau^{\geq
0}$ where $\tau^{\geq 0}(X)$ is the complex \ref{c3}.

Let $\A = D^{\leq 0}\cap D^{\geq 0}$ be the heart of the
$t$-structure. Thus objects of $\A$ are sequences
\begin{equation}\label{heartA}
[A \overset{f}{\to} B \overset{g}{\to} C]
\end{equation}
such that $f$ is injective and $\Ker g = \Jm f$. (We use the
square-brackets to stress that we consider an object in the
abelian category $\A$.) Morphisms in $\A$ are morphisms of
complexes up to homotopy.

Notice that $[A \to B \to C] \cong 0$ iff $B \to C$ is a split
surjection.

We next describe the kernels, images and cokernels in $\A$. Fix a
morphism
$$
\phi =(\phi_A, \phi_B, \phi_C): [A \to B \overset{f}{\to} C] \to
[A' \to B' \overset{f'}{\to} C']
$$

Let $M = M(\phi)$ be the mapping cone of $\phi$ (considered as a
morphism in $\KC$). Then we have $\Ker(\phi) = \tau^{\leq
0}(M[-1])$ and $\Coker(\phi) = \tau^{\geq 0}(M)$ and it follows
from this that
\begin{equation}\label{Ker2}
\Ker(\phi) = [\Ker \pi \to B \oplus \Ker d' \overset{\pi}{\to} C
\times_{C'} B']
\end{equation}
\begin{equation}\label{CoKer2}
\Coker(\phi) = [\Ker \delta \to C\oplus B' \overset{\delta}{\to}
C']
\end{equation}
\begin{equation}\label{Im2}
\Jm(\phi) = [A' \to (C \times_{C'} B') \overset{\pi}{\to} C]
\end{equation}
Here, $C \times_{C'} B' = \{(c,b');\, \phi_C(c) = -f'(b')\}$,
$p(b,b') = (-f(b),b'+\phi_B(b))$, $\delta = \phi_C + f'$ and
$\pi(c,b') = c$. Notice that in the case that $f'$ is surjective,
$\Jm \phi$ is just the pull-back of an exact sequence by $C$.

The localization functor $\Loc: \KC \to \DC$ induces a functor on
hearts
\begin{equation}\label{Lheart}
\Loc \vert_{\A}: \A \to \Cab
\end{equation}
given by $\Loc \vert_{\A}([A \to B \overset{d}{\to} C]) = \Coker d$.
$\Loc \vert_{\A}$ is exact since $\Loc$ is $t$-exact.
\begin{definition}\label{A 0} Let $\A^0$ be the full subcategory of $\A$ consisting of
short exact sequences.
\end{definition}
It follows from the exactness of $\Loc \vert_{\A}$ that $\A^0$ is
an exact abelian subcategory of $\A$.
\begin{definition}\label{functor P} Define a fully faithful
functor by $$\Pp: \Cab \to \A, \; A \mapsto \Pp_A := [0 \to 0 \to
A]$$
\end{definition}

Notice that for any object $[A \to B \to C]$ of $\A$ there is a
canonical exact sequence
\begin{equation}\label{standard resolution}
0 \to P_A \to P_B \to P_C \to [A \to B {\to} C] \to 0
\end{equation}
This shows that the homological dimension of $\A$ is $\leq 2$; it
is easy to see that strict inequality holds iff $\Cab$ is
semi-simple and in that case $\A$ is also semi-simple.

Let $\mathfrak{Ab}$ be the category of abelian groups and let
$\mathfrak{Ab}^{\Cab^{op}}$ be the abelian category of additive
contravariant functors from $\Cab$ to $\mathfrak{Ab}$. Then there
is the fully faithful Yoneda embedding $h: \Cab \to
\mathfrak{Ab}^{\Cab^{op}}$ defined by $A \mapsto h_A :=
\Hom_\Cab(A, \;)$. By the Yoneda lemma we have that $h_A \in
\Proj(\mathfrak{Ab}^{\Cab^{op}})$ for all objects $A$ of $\Cab$
and that $\mathfrak{Ab}^{\Cab^{op}}$ is generated by the
collection $\{h_A; A \in \Cab\}$.
\begin{proposition} There is an exact fully faithful functor $\pi: \A \to
\mathfrak{Ab}^{\Cab^{op}}$ such that $P_A \mapsto h_A$.
\end{proposition}
\begin{proof} Let $f: A \to B$ be a morphism in $\Cab$. Then we
must have $\pi([\Ker f \to A \to B]) = \Coker h_f$, where $h_f:
h_A \to h_B$ is given by $f$. By construction $\pi$ is exact and
by the Yoneda lemma we have that
$$
\Hom_\A(P_A,P_B) \cong \Hom_\Cab(A,B) \cong
\Hom_{\mathfrak{Ab}^{\Cab^{op}}}(h_A,h_B)
$$
for all objects $A,B$ in $\Cab$. Since the $P_A$ generates $\A$ as
a category the full faithfulness now follows from general
nonsense.
\end{proof}

\begin{cor}\label{K equivalent to D} For each $A \in \Cab$, $P_A \in \Proj(\A)$, and each projective of
$\A$ is isomorphic to some $P_A$. Hence, $\operatorname{D}^b(\A)$
is canonically equivalent to $\KC$.
\end{cor}
\begin{proof}
It follows from the previous proposition that $P_A$ is projective
in $A$ since $h_A$ is projective in $\mathfrak{Ab}^{\Cab^{op}}$.
Thus $\A$ has enough projectives. Moreover, it is easy to see that
each projective in $\A$ must be of the form. Thus, $\Proj(\A)
\cong \Cab$ and hence
$$\operatorname{D}^b(\A) \cong \Kb(\Proj(\A)) \cong \KC.$$\end{proof}

\begin{example}\label{the coinvariants} Let $\FF$ be a field, $R=\FF[x]/(x^2)$, $\Cab =
\Rmod$ and let $\A$ be the heart of the $t$-structure from
Proposition \ref{p1}. Then $\A$ has five indecomposable objects:
$P_{\FF}$ (projective), $P_R$ (projective and injective), $[0 \to
\FF \to R]$ (simple), $[\FF \to R \to \FF]$ (simple) and $[\FF \to
R \to R]$ (injective).
\end{example}

\subsection{Other standard $t$-structures on $\KC$}\label{other
standards} In general $\KC$ will have infinitely many standard
$t$-structures. It would be interesting to classify them (and also to relate them to Bridgeland's stability theory, \cite{B}, that
classifies all bounded $t$-structures).

In this section we make no attempt to reach such a classification, but merely observe that besides the one we studied in the previous section there
are two other particulary evident standard $t$-structures. We
denote them by $(D'^{\leq 0}, D'^{\geq 0})$ and $(D''^{\leq 0},
D''^{\geq 0})$ and their hearts by $\A'$ and $\A''$, respectively.
They are defined as follows:
\begin{equation}\label{leq0'}
D'^{\leq 0} = \{X \in \KC;\,X^i = 0 \hbox{ for } i > 1 \hbox{ and
}d^0 \hbox{ is surjective }\}
\end{equation}
\begin{equation}\label{geq0'}
D'^{\geq 0} = \{X \in \KC;\,  X^i = 0 \hbox{ for } i <-1  \hbox{
and } d^{-1} \hbox{ is injective }\}
\end{equation}
Its heart is given by complexes
\begin{equation}\label{'heart} \A' = \{[A
\overset{g}{\to} B \overset{f}{\to} C];\; g \hbox{ is injective
and } f \hbox{ is surjective}\}
\end{equation}
The other one is defined by
\begin{equation}\label{leq0''}
D''^{\leq 0} = \{X \in \KC;\,X^i = 0 \hbox{ for } i > 2 \hbox{ and
} H^1(X) = H^2(X) = 0\}
\end{equation}
\begin{equation}\label{geq0''}
D''^{\geq 0} = \{X \in \KC; X^i = 0, \hbox{ for } i < 0\}
\end{equation}
\begin{equation}\label{''heart} \A'' =\{[A \overset{g}{\to} B
\overset{f}{\to} C];\; f \hbox{ is surjective and } \Ker f = \Jm
g\}
\end{equation}
The description of the $t$-structure $(D''^{\leq 0}, D''^{\geq
0})$ is dual to that of $(D^{\leq 0}, D^{\geq 0})$: The objects
$[A \to 0 \to 0]$ are injective in $\A''$, for $A \in \Cab$, and
the functor $A \mapsto [A \to 0 \to 0]$ corresponds to the Yoneda
embedding $\Cab \to ({\mathfrak{Ab}^{\Cab}})^{op}$ defined by $A
\mapsto \Hom_\Cab(\ ,A)$.

The embedding $\Cab \to \A'$, $A \mapsto [0 \to A \to 0]$ appears
to be a mix of the two Yoneda embeddings. We shall not investigate
these two $t$-structures any further in this paper.

We conclude this section with an example that shows there are many
standard $t$-structures on $\Cab$.
\begin{example}\label{manyt} Let $\FF$ be a field, $R=\FF[x]/(x^2)$ and let $\Cab = \Rmod$. Let $n \geq 2$ and let $V$ denote the acyclic complex
$$
0 \to \FF \hookrightarrow R \overset{x}{\to} R \overset{x}{\to}
\ldots \overset{x}{\to} R \twoheadrightarrow \FF \to 0
$$
where the component $\FF$ occurs in degree $0$ and $n$.  Let
$\Omega$ denote the set of all complexes in $\KR$ concentrated in
degrees $\leq 0$ together with the complex $V$.

It follows e.g. from \cite{Ay}, Proposition 2.1.70, that there is
a unique $t$-structure $(D^{\leq 0}_n, D^{\geq 0}_n)$ on $\KR$
such that $\Omega \subset D^{\leq 0}_n$ and
$$D^{\geq 0}_n = \{X \in \KR;\, \Hom_{\KR}(A[i], X) = 0, A
\in \Omega, i > 0\}.
$$
One can verify that this gives standard $t$-structures which are
different for different values of $n$.
\end{example}

\section{Injectives}
We describe the injectives in $\A$ and in $\A^0$. Contrary to the
case with projectives, in order for $\A$ to have enough injectives
we need that $\Cab$ has enough injectives.
\subsection{Injectives in $\A$}
Let us start with
\begin{lemma}\label{epimono} Let $$\phi =(\phi_A, \phi_B, \phi_C): [A \to B \overset{f}{\to} C] \to
[A' \to B' \overset{f'}{\to} C']$$ be a morphism in $\A$.
\textbf{i)} Assume that $[A' \to B' \overset{f'}{\to} C'] \neq 0$
and that $\End_{\Cab}(C')$ is a local ring. Then $\phi$ is
surjective iff $\phi_C$ is a split surjection. \textbf{ii)} $\phi$
is injective iff $\bar{\phi}_C: C/\Jm f \to C' /\Jm f'$ is
injective and the canonical injection $A \to B \oplus \Ker f'$, $a
\to (a,-\phi_Ba)$ splits.
\end{lemma}
\begin{proof} \textbf{i)} By \ref{CoKer2} we have that $\phi$ is
surjective iff $\phi_C + f': C\oplus B' \to C'$ is a split
surjection. Since $\End_{\Cab}(C')$ is a local ring this implies
that either $f'$ or $\phi_C$ is a split surjection. By the
assumption $f'$ is not a split surjection. Hence $\phi_C$ is a
split surjection.

\smallskip

\noindent Let us prove \textbf{ii)}. By \ref{Ker2} we have that
$\phi$ is injective iff
\begin{equation}\label{eneq}
B \oplus \Ker f' \overset{\epsilon}{\longrightarrow} C \times_{C'}
B'
\end{equation}
is a split surjection, where $\epsilon(b,v) = (-fb, \phi_B b +v)$,
for $(b,v) \in B \oplus \Ker f'$.

\noindent \textbf{Claim:} $\epsilon$ is surjective iff
\begin{equation}\label{claimconcl}
\bar{\phi}_C: C/\Jm f \to C'/ \Jm f' \hbox{ is injective.}
\end{equation}

\emph{Proof Claim.} Note that \ref{claimconcl} is equivalent to
\begin{equation}\label{concl}
\phi_C c \in \Jm f' \implies c \in \Jm f
\end{equation}

Denote by $K$ the righthand side of \ref{eneq} and fix $(c,b') \in
K$. Thus $\phi_C c = -f'b'$ and so if we assume that $\epsilon$ is
surjective we see that \ref{concl} holds. Conversely, assuming
\ref{concl} we show that $\epsilon$ is surjective. We have $(c,b')
= (fb,b')$ for some $b \in B$. Then $f'b' = -\phi_Cc = -\phi_C fb
= -f' \phi_B b$. Let $v = b'+\phi_B b \in \Ker f'$. Then we see
that $(fb,b') = \epsilon(b,v)$. This proves the claim.

Now, for $\epsilon$ surjective, we have that $\epsilon$ splits iff
the inclusion
\begin{equation}\label{Kerepsilon}
\Ker \epsilon = \{(b,-\phi_Bb); b \in Ker f\} \hookrightarrow B
\oplus \Ker f'
\end{equation}
splits which proves \textbf{ii)}.
\end{proof}

We can now prove
\begin{proposition}\label{structure of injectives} An object $[D \to I \to J]$ in $\A$ is injective if
$I, J \in \Inj(\Cab)$.
\end{proposition}
\begin{proof} \textbf{a)} Assume that $I \in \Inj(\Cab)$.
We first show that $P_I$ is injective in $\A$. Consider a
commutative diagram
\begin{equation}\label{diagraminjective}
   \begin{diagram}
    \node[2]{P_I} \\
  \node{0} \arrow{e}  \node{[A \to B
\overset{f}{\to} C]} \arrow{n,r}{\gamma} \arrow{e,b}{\phi}
\node{[A' \to  B' \overset{f'}{\to} C']}
\arrow{nw,t,..}{\tilde{\gamma}}
   \end{diagram}
\end{equation} We must construct a lift $\tilde{\gamma}$. By
Lemma \ref{epimono} the map $\overline{\phi}_C: C/ \Jm f \to C' /
\Jm f'$ is injective and moreover we see that $\gamma$ factors
through $P_{C/\Jm f}$. Hence we can fill the dotted arrow and get
a commutative diagram
   $$
   \begin{diagram}
    \node[2]{P_I} \\
    \node[2]{P_{C/\Jm f}}\arrow{n,r}{(0,\overline{\gamma}_C)}\arrow{eee,b}{\overline{\phi}_C}    \node{P_{C'/\Jm f'}}
    \arrow{nw,t,..}{(0,\widetilde{\bar{\gamma}}_C)}  \\
    \node{[A \to B\to C]} \arrow[3]{eee,b}{\phi} \arrow{nne,l}{\gamma} \arrow{ne,l}{nat}
    \arrow{e}\node[3]{[A' \to B' \to
    C']} \arrow{nw,l}{nat'}
   \end{diagram}
   $$
where $nat$ and $nat'$ are the natural maps. Hence we can take
$\tilde{\gamma} = (0, \widetilde{\bar{\gamma}}_C \circ nat')$.

\smallskip

\noindent \textbf{b)} We have the exact sequence
\begin{equation}\label{}
0 \to P_{D} \to P_I \to P_J \to [D \to I \to J] \to 0
\end{equation}
By \textbf{a)} $P_I$ and $P_J$ are injective. Since the
homological dimension of $\A$ is $\leq 2$ it follows that $[D \to
I \to J]$ is injective as well.
\end{proof}

\begin{cor}\label{enough injectives}  If $\,\Cab$ has enough injectives
then also $\A$ has enough injectives.
\end{cor}
\begin{proof}   Let $[A \to B \overset{d}{\to} C]$ in $\A$ be
given. By standard arguments we can find an object $[A' \to I
\overset{\pa}{\to} J]$ with $I,J$ injective and a morphism $\phi:
[A \to B \overset{d}{\to} C] \to [A' \to I \overset{\pa}{\to} J]$
with the properties that $\phi_B$ and $\phi_C$ are injective,
$\phi_B(A) = A'$ and $\overline{\phi}_C: C / \Jm d \to J / \Jm
\pa$ is injective. Thus by Lemma \ref{epimono} $\phi$ is
injective.
\end{proof}
 The converse of Proposition \ref{structure of injectives} also
 holds
\begin{cor}\label{structure of injectives cor}  Assume that $\Cab$ has enough injectives. Each injective object of $\A$ is isomorphic to an
object of the form $[A \to I \to J]$, where $I,J \in \Inj(\Cab)$.
\end{cor}
\begin{proof} Let $E \in \A$ be injective. In the proof of
corollary \ref{enough injectives} we saw that we can find an
object $[A' \to I' \to J']$, with $I'$ and $J'$ injective in
$\Cab$ and an injective morphism $\phi: E \to [A' \to I' \to J']$.
Since $E$ is injective $\phi$ splits. This implies that $E$ has
the desired form.
\end{proof}
Note that corollary \ref{structure of injectives cor} gives a
natural bijection $\Ind(\Cab) \cong \Ind(\Inj(\A))$: a
non-injective $A \in \Cab$ corresponds to $[A \to I \to J]$, where
$A \hookrightarrow I \to J$ is an indecomposable injective
resolution of $A$ and an injective $A \in \Cab$ corresponds to $[0
\to 0 \to A]$.

\subsection{Injectives and projectives in $\A^0$}
Recall that $\A^0$ denotes the full subcategory of $\A$ whose
objects are short exact sequences. The inclusion $\A^0 \to \A$ has
the right adjoint
\begin{equation}\label{functor q}
q: \A \to \A^0, \ q([A \to B \overset{d}{\to} C]) = [A \to B \to
\Jm d]
\end{equation}
Assume that $\Cab$ has enough projectives and injectives and put
for $X \in \Cab$, $P^0_X = [\Ker d \to Q \overset{d}{\to} X]$
where $d:Q \to X$ is surjective and $Q$ is projective. Also, put
$I^0_X = q(I_X)$. Then we have
\begin{lemma}\label{projs and injs in A0}  For each $X \in \Cab$, $P^0_X$ is projective in
$\A^0$ and $\Hom_{\A^0}(P^0_X, V) = \Hom_\A(P_X,Y)$ for $V \in
\A^0$. Any indecomposable projective in $\A^0$ is (isomorphic to
an object) of the form $P^0_X$ where $X$ is indecomposable and
non-projective in $\Cab$. Similarly, each $I^0_X$ is injective in
$\A^0$, $\Hom_{\A^0}(V, I^0_X) = \Hom_\A(V,I_X)$ and each
indecomposable injective in $\A^0$ is of the form $I^0_X$ where
$X$ is indecomposable and non-injective in $\Cab$.
\end{lemma}
\begin{proof} All verifications are left to the reader. For the
part which states that all projectives and injectives in $\A^0$
are isomorphic to objects of the prescribed form just mimic the
argument of corollary \ref{structure of injectives cor}.
\end{proof}

\section{AR-sequences for representations of an artin algebra}

\subsection{AR-sequences are simple objects of
$\A$}\label{simple object section} In this section we assume that
$\Cab$ is a finite length (abelian) category. We start by briefly
recalling AR-theory in $\Cab$, (see \cite{ARS} for details about
the material here and compare with \cite{H} and \cite{K} for the
theory of AR-triangles in triangulated categories that is not
treated here). Fix a morphism
$$
B \overset{f}{\to} C
$$
in $\Cab$. $f$ is called an almost split right map if $f$ is
\emph{not} a split surjection and any map $\phi: X \to C$ which is
not a split surjection factors through $f$. Assume from now on
that $f$ is right almost split. It follows that $C$ is necessarily
indecomposable.

Almost split right maps have the following properties:

\begin{itemize}
    \item If $C$ is projective it has a unique maximal submodule
    $rad\,
C$ and $f$ is the inclusion $rad\, C \hookrightarrow C$.
    \item If $C$ is not projective then $f$ is necessarily
    surjective.
\end{itemize}
Dually, there is the notion of an almost split left map $g:A \to
B$. $g$ is not a split injection and any $h: A \to Y$ which is not
a split injection factors through $g$.

A short exact sequence
$$
0 \to A \overset{g}{\to} B \overset{f}{\to} C \to 0
$$
is called an almost split exact sequence, or an AR-sequence, if
$g$ is left almost split and $f$ is right almost split. See
\cite{A}, \cite{J} and \cite{Sm} for some positive and negative
existence results for AR-sequences.

Let $h: X \to Y$ in $\Cab$ be a given map in $\Cab$. Recall that a
(right) minimal version of $h$ is a map $h_{min}: X' \to Y$ such
that $h_{min}$ factors through $h$ and $h$ factors through
$h_{min}$ and $X'$ has minimal length with this property.
$h_{min}$ exists and is unique up to isomorphism of maps over $Y$.
The minimal length of $X'$ is equivalent to require that $X'$ has
no non-zero direct summand mapped to $0$ by $h_{min}$.

If one assumes that $h$ is right almost split it follows that
$\Ker h_{min}$ is indecomposable.

It is easy to see that if $B \overset{f}{\to} C$ and $B'
\overset{f'}{\to} C$ are almost split right maps, then $[\Ker f
\to B \overset{f}{\to} C] \cong [\Ker f' \to B' \overset{f'}{\to}
C]$ in $\A$. We next show that the almost split right maps are
precisely the simple objects of $\A$.

\begin{proposition}\label{classification} Let
$X = [A \to B \overset{d}{\to} C]$ be an object of $\A$ and assume
that $C$ be indecomposable in $\Cab$. Then $X$ is simple iff $d$
is an almost split right map. If $d$ is almost split we write $L_C
:= X$. In this case $L_C$ is the unique simple quotient of $P_C$.
\end{proposition}
\begin{proof}
Note that by the Krull Schmidt theorem $\End_{\Cab}(C)$ is a local
ring.

Assume that $B \overset{d}{\to} C$ is almost split. We shall show
$X$ is simple. For this it is enough to show the following: Let
$\phi : X \twoheadrightarrow Y$, with $Y \neq 0$, be a surjective
map. Then $\phi$ is injective.

We may assume that $d$ is right minimal. Then any endomorphism $h$
of $B$ over $C$ is an automorphism. Since $C$ is indecomposable we
may assume that $Y = [A' \to B' \overset{d'}{\to} C]$, where $d'$
is not a split surjection, and that $\phi_C = \Id_C$. Since $d$ is
almost split it follows that $d' = d \circ g$, for some $g: B' \to
B$. Then $g \circ \phi_B$ is an endomorphism of $B$ over $C$ and
hence an isomorphism. This means that $\phi \circ (g, \Id_C)$ is
an automorphism of $L_C$. Thus $\phi$ is injective.

Conversely, assume that $d$ is not almost split. Then we can find
$f:D \to C$ which is not a split epi such that $f$ does not factor
through $d$. Consider the composition
$$
P_D \to X \twoheadrightarrow [\Ker (d+f) \to B \oplus D
\overset{d+f}{\to} C]
$$
The third object and the first map are non-zero by assumption. But
the composition is zero, so $X$ is not simple.

For the last assertion, we have already proved that $L_C$ is
simple and, clearly, $L_C$ is a quotient of $P_C$. Conversely, if
$Y$ is a simple quotient of $P_C$ it follows from Lemma
\ref{epimono} \textbf{i)} that we may assume the end-term of $Y$
is $C$. Then by what we just have shown we see that $Y$ is given
by an almost split right map with target $C$, i.e. $Y \cong L_C$.
\end{proof}
We next reprove the well-known result that right almost split maps
fit into AR-sequences:
\begin{cor}\label{leftsplit} Assume that $\Cab$ has enough injectives.
Let $0 \to A  \overset{g}{\to} B \overset{f}{\to} C \to 0$ be a
short exact sequence in $\Cab$ such that $f$ is minimal right
almost split. Then $g$ is left almost split.
\end{cor}
\begin{proof} Recall that $A$ is indecomposable by the minimality
of $f$. Let $X \in \Ob(\Cab)$ and a non-split injection $h: A \to
X$ be given. We must prove that $h$ factors through $B$. We may
assume that $h$ is non-zero. Thus $h$ is a non-split map.

Next, by the assumption that $f$ is right almost split we know
that $L_C = [A \to B \overset{f}{\to} C]$ is a simple object of
$\A$. Let
$$
0 \to X \to I \overset{j}{\to} J
$$
be an injective resolution of $X$. Then there is a map $\tilde{h}:
L_C \to [X \to I \overset{j}{\to} J]$ such that $\tilde{h}_A = h$.
We claim that $\Ker \tilde{h} \neq 0$: indeed, if $\Ker \tilde{h}
= 0$ then by Lemma \ref{epimono} $A \to B \oplus X$ would be a
split injection and this is not the case since $A$ is
indecomposable and neither $A \to B$ nor $h$ is split.

Hence, by simpleness of $L_C$ we have that the natural map $nat
:\Ker \tilde{h} \to L_C$ is an isomorphism. One sees that the
inverse morphism $nat^{-1}$ provides a map $h': B \to X$ such that
$h' \circ g = h$. \end{proof}

To end this section let us say that an abelian category has enough
simples if each of its objects has a simple quotient object. A
noetherian category of course has enough simples. The following
example shows however that $\A$ will in general not be noetherian
although it (by Auslander and Reiten's theorem) has enough
simples. Better means that guarantee enough simples will be given
in the following sections.

\begin{example} Let $\FF$ be a field and let $R =
\FF[x,y]/(x^2,xy,y^2)$ and $\Cab = \Rmod$. Let $\m = (x,y)$ be the
maximal ideal in $R$.

For $i >0$ define $R$-modules $M_i = R/x-iy$ and $B_i = M_1 \oplus
M_2 \oplus \ldots \oplus M_i$. Let $\FF = R/\m$ and $B_i \to \FF$
the sum of the natural projections. Let $V_i = [Ker_i \to B_i \to
\FF]$, where $Ker_i = \Ker(B_i \to \FF)$. The inclusions $B_i
\hookrightarrow B_{i+1}$ induce surjections
\begin{equation}\label{refsurBi}
V_1 \twoheadrightarrow \ldots \twoheadrightarrow V_i
\twoheadrightarrow V_{i+1} \twoheadrightarrow \ldots
\end{equation}
 Since $\Hom_{\A}(P_{M_{i+1}},
V_i) \neq 0$ and  $\Hom_{\A}(P_{M_{i+1}}, V_{i+1}) = 0$ we
conclude that non of the maps in \ref{refsurBi} are isomorphisms
and hence that d.c.c. doesn't hold on quotient objects in $\A$, so
$\A$ is not noetherian.
\end{example}
\subsection{Serre duality for the category $\A$ in the case of representations of an artin algebra.}

Let $R$ be an artin algebra. Thus, by definition there is a
commutative artin ring $\F \subset R$ such that $R$ is a finitely
generated $\F$-module. From now on we shall exclusively consider
the case where $\Cab = \Rmod$.

Let $S$ be the direct sum of the irreducible $\F$-modules and let
$J$ be an injective hull of $S$ in $\kmod$. Let
$$
\kMod \ni M \mapsto M^*:= \Hom_\F( M ,J) \in \kMod
$$
be the usual duality functor. Thus ${}^{**} \vert_{\kmod} \cong
\Id_{\kmod}$.

In order to later on apply Watts' representability theorem
(\cite{R}, Theorem 3.36) we need to embed $\A$ into a category of
modules over a ring. Some technical difficulties arise from the
fact that $\A$ does not have a small projective generator (unless
$R$ has finite representation type). Let $\widetilde{\A} :=
\A(\RMod)$. Thus, $\widetilde{\A}$ is an abelian category closed
under coproducts containing $\A$ as a full abelian subcategory.
Since $\Ind(\Rmod)$ is a set we can define $\widetilde{P} :=
\oplus_{D \in \Ind{(\Rmod)}} P_D \in \widetilde{\A}$ and $F =
\End_{\w\A}(\w P)$.

Consider the exact functor
$$\V: \w\A \to \ModF, \, \V M = \Hom_{\w\A}(\w P, M)$$
and let $\VA : \A \to \modF$ be the restriction of $\V$ to $\A$.
\begin{lemma}\label{Morita lemma} $\VA$ is a fully faithful
embedding. The essential image of $\VA$ consists of all objects
$Y$ in $\text{mod-}F$ such that there exists a $C \in \Rmod$ and a
surjection $\V P_C \to Y$.
\end{lemma}
\begin{proof} Let us say
that an object $M$ in $\w \A$ is good if the natural map
$$
\Hom_{\w \A}(M,N) \to \Hom_{\ModF}( \V M , \V N)
$$
is bijective for all $N$ in $\A$. We must show that any object $M$
in $\A$ is good. First take $M = P_C$ for $C \in \Ind (\Rmod)$.
Since, $P_C$ is a direct summand in $\w P$ we get that $\V P_C$ is
a direct summand in the right $F$-module $F = \V \w P$. From this
it is clear that $P_C$ is good. Thus $P_C$ is good for any $C \in
\Rmod$. Since every object $M$ in $\A$ has a presentation $P_C \to
P_{C'} \twoheadrightarrow M$ we get by the five lemma that $M$ is
good. This proves the full faithfulness of $\V |_\A$.

Since any object $M$ in $\A$ is a quotient of some $P_C$ it
follows that $\V M$ is a quotient of $\V P_C$. Conversely, assume
that $Y \in \modF$ and there is a surjection $\phi: \V P_C \to Y$.
Let $T$ be the $\A$-subobject of $P_C$ defined by
$$
T = \sum_{g \in \V P_C} \Jm g
$$
We now prove $\Ker \phi = \V T$. Note that if $I$ is a
sufficiently large index set we get a surjection $f: {\w P}^I \to
T$ such that each component $f_i : \w P \to P_C$ is in $\Ker
\phi$. Let $T = [B' \to B'' \to B]$ for some $B', B'', B \in
\Rmod$. We may assume that $B$ has no direct summand which is the
isomorphic image of a direct summand of $B''$.

Since $f$ is surjective we have that the natural map $B'' \oplus
{\w P}^I {\to} B$ is a split surjection. If $B$ is indecomposable
the Krull Schmidt theorem shows that some component $f_j$ of $f$
is a split surjection. In general, after breaking $B$ into
indecomposable pieces we find that there is a finite subset $J
=\{1, \ldots , n\} \subset I$ such that $f_J := f |_{{\w P}^J} :
{\w P}^J \to B$ is a split surjection.

Since $\w P$ is projective in $\A$ we see that any map $g: \w P
\to T$ factors as $g = f_J \circ h$ where $h = (h_1, \ldots , h_n)
\in \Hom_{\w \A}(\w P, {\w P}^J) = \Hom_{\w \A}(\w P, {\w P})^n$.
Thus $g = \sum^n_{i=1} f_i h_i \in \Ker \phi$. Thus $\Ker \phi =
\V T$. Hence we have a short exact sequence
$$0 \to \V T \overset{\mu}{\to} \V P_C \to Y \to 0$$
By the full faithfulness already proved we see that $\mu = \V \nu$
for some $\nu \in \Hom_{\A}(T, P_C)$. Hence, $Y \cong \V(\Coker
\nu)$.
\end{proof}

\begin{proposition}\label{Serre duality} For any $X \in R\text{-\hbox{mod}}$, the contravariant functor
$$
\Hom_{{\A}}(P_X, \ )^*: \A \to \kmod
$$
is representable by an injective $\A$-object $S P_X$.
\end{proposition}
We shall refer to $SP_X$ as the Serre dual of $P_X$.
\begin{proof}
Consider the contravariant functor
$$\Gamma := \Hom_{\modF}(\V P_X, \ )^*: \modF \to \kMod
$$
Since the right $F$-module $\V P_X$ is a direct summand in $F$, if
$X$ is indecomposable, we conclude that $\V P_X$ is projective for
any $X$ in  $R\text{-mod}$. Thus $\Gamma$ is exact. In particular
$\Gamma$ is left exact and since, moreover, $\Gamma$ transforms
coproducts to products, Watts' theorem shows that $\Gamma$ is
represented by $\Gamma(F)$.

By Lemma \ref{Morita lemma} it remains to show that there is an
object $SP_X \in \A$ such that $\Gamma(F) = \V S P_X$. Note that
$$
\Gamma(F) \cong \Hom_{\w \A}(P_X,\w P)^*
$$
with the right $F$-module structure on $\Hom_{\w\A}(P_X,\w P)^*$
induced by the left $F$-module structure on $\Hom_{\w\A}(P_X,\w
P)$ that is given by composition of maps.

Assume first that $X$ is irreducible and let $I$ be an injective
hull of $X$ in $R\text{-mod}$. Let $\epsilon_1, \ldots ,
\epsilon_n$ generate $\Hom_{\A}(P_X,P_I)^*$ as a $\F$-module.
Since $P_X$ and $P_I$ are direct summands in $\w P$ we can
interpret $\End_\A(P_X)^*$ and $\Hom_{\A}(P_X,P_I)^*$ as direct
summands in $F^*$. With this in mind we record that
\begin{equation}\label{span property}
\End_{\A}(P_X)^* = \{\sum^n_{i=1} \epsilon_i g_i; \,  g_i \in
\Hom_\A(P_X, P_I)\}
\end{equation}

Similarly, we can consider the $\epsilon_i$'s as elements of
$\Hom_{\w\A}(P_X,\w P)^*$ and get the map
$$
\pi: (\V P_I)^n \to \Hom_{\w\A}(P_X,\w P)^*, \, (f_1, \ldots ,f_n)
\mapsto \sum \epsilon_i f_i
$$
Note that $\pi$ is right $F$-linear. We shall prove that $\pi$ is
surjective. Let
$$\nu = \prod \nu_D \in \Hom_{\w \A}(P_X,\w P)^* = \prod_{D \in
\Ind (\Rmod)} \Hom_\A(P_X, P_D)^*
$$
be given.

Fix for now $D \in \Ind (R\text{-mod})$. We can write $\soc D =
X^m \oplus K$ where $K$ is a direct sum of simple modules all of
them non-isomorphic to $X$, where $\soc D$ is the socle of $D$.
Thus, since $X$ is simple, $\Hom_\A(P_X, P_D) \cong \Hom_\A(P_X,
P_{X^m})$ and so we have an isomorphism
$$
nat: \Hom_\A(P_X, P_D)^* \cong \Hom_\A(P_X, P_{X^m})^* =
(\Hom_\A(P_X, P_{X})^*)^m
$$
By \ref{span property} we can find $h_{D,1}, \ldots h_{D,n} \in
\Hom_\A(P_{X}, P_{I})^m$ such that $nat(\nu_D) = \sum^n_{i=1}
\epsilon_i h_{D,i}$. Here we used the notation $h_{D,i} =
(h_{D,i1}, \ldots , h_{D,im})$ and $\epsilon_i h_{D,i} =
(\epsilon_i h_{D,i1}, \ldots, \epsilon_i  h_{D,im})$.

Since $P_I$ is injective we can find $\tilde{h}_{D,i} \in
\Hom_\A(P_{D}, P_{I})^m$ that extends $h_{D,i}$. Thus ${\nu_D} =
\sum^n_{i=1} \epsilon_i \tilde{h}_{D,i}$. If we let $h_D =
(\tilde{h}_{D,1}, \ldots , \tilde{h}_{D,n})$ and $h = \prod_{D \in
\Ind (R\text{-mod})} h_D$ we see that $\pi(h) = \nu$.

For $X \in \Rmod$ simple we have now constructed a surjection $\V
P_A \to \Hom_{\w \A}(P_X, \w P)^*$, with $A = I^n$. If $X$ is not
simple, it has a finite filtration with simple subquotients $X_1 ,
\ldots , X_N$. By the procedure above we find $A_i \in \Rmod$ and
surjections $\V P_{A_i} \to  \Hom(P_{X_i}, \w P)^*$. Since each
$\V P_{A_i}$ is projective in $\modF$ this give rise to a
surjection
$$
\oplus^N_{i=1} \V P_{A_i} \to \Hom_{\w \A}(P_X , \w P)^*
$$
Thus, by Lemma \ref{Morita lemma} there is an object $SP_X \in \A$
such that $\Hom_{\w \A}(P_X , \w P)^* \cong \V SP_X$.
\end{proof}
Note that the assignement $P_X \mapsto SP_X$ defines a fully
faithful functor $\Proj(\A) \to \Inj(\A)$, because for $P_X,P_Y
\in \Proj(\A)$ we have isomorphisms
$$
\Hom_\A(P_X,P_Y) \to \Hom_\A(P_X,P_Y)^{**} \to \Hom_\A(P_Y,
SP_X)^*\to
$$
$$
\Hom_\A(SP_X, SP_Y)
$$
Let us explicitly describe $Sf: SP_X \to SP_Y$ corresponding to $f
\in \Hom_\A(P_X, P_Y)$. $f$ induces a morphism
\begin{equation}\label{explicit}
\Hom_\A(f,\w P)^*: \V SP_X \cong \Hom_\A(P_X, \w P)^* \to
\Hom_\A(P_Y, \w P)^* \cong  \V SP_X
\end{equation}
Now the full faithfulness of $\V |_\A$ gives a morphism $Sf: SP_X
\to SP_Y$ such that $\V Sf = \Hom_\A(f,\w P)^*$.

For exactness reasons there cannot exist a Serre dual object $SA
\in \A$ of a non-projective object $A \in \A$. But we have
\begin{proposition}\label{Serre dual on Kb}
The functor $S: \Proj(\A) \to \Inj(\A)$ induces a triangulated
functor $\Ss: \DA \to \DA$ satisfying
$$
\Hom_{\DA}(A,B)^* \cong \Hom_{\DA}(B,\Ss A)
$$
for all $A,B \in \DA$.
\end{proposition}
\begin{proof} We have that $S$ extends to a functor
$$
\Ss: \Cb(\Proj(\A)) \to \Cb(\Inj(\A))
$$
Clearly, this functor induces a triangulated functor between
homotopy categories that we denote by the same symbol
$$
\Ss: \Kb(\Proj(\A)) = \DA \to \Kb(\Inj(\A)) = \DA
$$
For $A,B \in \operatorname{C}^b(\A)$ there is the homomorphism
complex $\HHom(A,B)$ defined by
$$
\HHom(A,B)^n = \prod_{i \in \Z}\Hom_\A(A_i,B_{i+n}), \text{ for }
n \in \Z
$$
and differential given by $df = d_B \circ f - (-1)^n f \circ d_A$
for $f \in \HHom(A,B)^n$. Using \ref{explicit} it is easy to
verify that for $A, B \in \operatorname{C}^b(\Proj(\A))$, the
already constructed isomorphisms $\Hom_\A(A_i,B_j)^* \cong
\Hom_\A(B_j,SA_i)$, $\forall i,j \in \Z,$ defines an isomorphism
of homomorphism complexes
$$
\HHom(A,B)^* \cong \HHom(B,\Ss A)
$$
Then we get $\Hom_{\DA}(A,B)^* = H^0(\HHom(A,B)^*)$, since $A \in
\operatorname{C}^b(\Proj(\A))$, and $\Hom_{\DA}(B,\Ss A) =
H^0(\HHom(B,\Ss A))$, since $\Ss A \in
\operatorname{C}^b(\Inj(\A))$. This establishes the isomorphism
stated in the theorem.
\end{proof}
Let us remark that using the $t$-structure it is easy to see that,
conversely, Proposition \ref{Serre dual on Kb} implies Proposition
\ref{Serre duality}. The result of Proposition \ref{Serre dual on
Kb} will not be used in this paper.
\subsection{Existence of AR-sequences for
representations of an artin algebra}

We now approach the existence problem of AR-sequences. First, it
is better to work with the category $\A^0$, since the
indecomposable projectives of $\A^0$ correspond to non-projectives
of $\Rmod$ and AR-sequences must have non-projective end term.
Write
$$SP^0_X :=
q(SP_X),$$ where $q$ is the functor \ref{functor q}. It follows
from Proposition \ref{Serre duality} and Lemma \ref{projs and injs
in A0} that
\begin{cor}\label{AL corr in good language} For $V \in \A^0$ we have a natural isomorphism $$\Hom_{\A^0}(P^0_X, V )^* \cong \Hom_{\A^0}(V ,
S P^0_X).$$
\end{cor}
It seems to be an appropriate way (which also generalizes well) to
think of AR-duality like this, as a Serre type of duality between
$\Proj(\A)$ and $\Inj(\A)$ (or between $\Proj(\A^0)$ and
$\Inj(\A^0)$). We can now deduce Auslander and Reiten's famous
existence theorem and in addition give a rather explicit form for
the AR-sequence with given end term.

\begin{theorem}\label{theorem main} Let $X \in \Ind(\Rmod)$ be non-projective. There exist a non-zero map $\tau :
P^0_X \to S P^0_X$ which has the property that any non-surjective
map $g: P^0_X \to \Jm \tau$ must satisfy $g = 0$. Then $\Jm \tau$
is an AR-sequence with end term $X$.
\end{theorem}
\begin{proof} We first show the existence of a non-zero map $\tau: P^0_X \to S
P^0_X$ such that $\tau \circ f = 0$ for all non-units $f$ in the
local artin algebra $\End_{\A^0}(P^0_X) = \End_{\Rmod}(X)$. Since
$X$ is non-projective we have $P^0_X \neq 0$ and thus by
AR-duality
$$\Hom_{\A^0}(P^0_X , S P^0_X) \cong \End_{\A^0}(P^0_X)^*  \neq 0$$
Now $\Hom_{\A^0}(P^0_X , S P^0_X)$ is a finitely generated
$\F$-module and therefor finitely generated as a right module over
the ring $\End_{\A^0}(P^0_X) \cong \End_{\Rmod}(X) \supset \F$. We
can thus find a nonzero element $\tau$ in the socle of
$\Hom_{\A^0}(P^0_X , S P^0_X)$ considered as an
$\End_{\A^0}(P^0_X)$-module.

This $\tau$ will satisfy the hypothesis of the theorem because any
map $g: P^0_X \to \Jm \tau$ will factor as $g = \tau \circ f$, for
some $f \in \End_{\A^0}(P^0_X)$, by the projectivity of $P^0_X$.
If $g$ is non-surjective it is clear that $f$ can not be a unit.
Thus $g = 0$ in this case.

We now prove that $\Jm \tau$ is simple. For this purpose it is
enough to prove that for any $D \in \RMod$ and any non-surjective
map $h: P^0_D \to \Jm \tau$ we must have $h = 0$. We have
$$
h = 0 \iff \Hom_{\A^0}(\Jm h, SP^0_X) = 0
\overset{\hbox{AR-duality}}{\Longleftrightarrow}
\Hom_{\A^0}(P^0_X, \Jm h) = 0
$$
Let $g \in \Hom_\A(P^0_X, \Jm h)$. Let $i: \Jm h \to \Jm \tau$ be
the inclusion. Since $i$ is not surjective $i \circ g: P^0_X \to
\Jm \tau$ is not surjective and hence $0$ by the assumption on
$\tau$. Thus, $g = 0$.

Finally we observe that $\Jm \tau$ is an AR-sequence that ends
with $X$: Since $SP_X$ is injective in $\A$ we see that $S P^0_X =
[X' \to I \to N]$ where $I$ is injective in $\Rmod$. Then we have
$ \Jm \tau = [X' \to I \times_N X {\to} X]$
\end{proof}

To end this section let us deduce the classical formulation of
Auslander and Reiten duality involving the dual of the transpose.
Let $X \in \Rmod$. Since $S P_X$ is injective we see that
\begin{equation}\label{Xprime}
SP_X \cong [X' \to I \overset{d}{\to} J]
\end{equation}
where $X' \in \Rmod$ and $I,J \in \Inj(\Rmod)$. We assume that
$X'$ is chosen to be minimal in the sense that it has no direct
summand $X''$ such that the composition $X'' \to X' \to I$ splits;
then $X'$ is well-defined up to isomorphism.

Then it is easy to see that for any $V \in \A^0$ one has
$$\Hom_{\A}(V, SP_X) = \Hom_{\KR}(V,X'[2]).$$
On the other hand (for any $V \in \A$) we have
$$\Hom_{\A}(P_X,V) = \Hom_{\KR}(X[0],V).$$
Thus we have rediscovered
\begin{proposition}\label{dualtranspose} (AR-duality.) There is a natural isomorphism $$\Hom_{\KR}(X[0],V)^*
\cong \Hom_{\KR}(V,X'[2])$$ of $\F$-modules, for $V \in \A^0$, $X
\in \Rmod$ and $X'$ defined by \ref{Xprime}.
\end{proposition}
In \cite{ARS} it is proved that the formula in Proposition
\ref{dualtranspose} holds with $X'$ replaced by the dual of the
transpose of $X$, $D Tr X$. Thus we have proved that $X' \cong D
Tr X$ and also reestablished the existence of the dual of the
transpose.

\section{Generalized AR-sequences}\label{Generalized
AR-sequences} Here we propose a generalization of AR-sequences to
the case of non-abelian categories. Let $\Cab$ be an additive
Karoubi closed category. Let $D^{\leq 0}$ be the subcategory of
$\KC$ defined by \ref{leq0}. Let $\tau^{\leq 0}:
\KC \to  D^{\leq 0}$ be the functor given by
\ref{form1}. Let $D^{>0}$ be the collection of $M \in \KC$
such that there is an $M' \in \KC$ such that $M$ is homotopic to
the cone of the canonical morhism $\tau^{\leq 0} M' \to M'$. We make the rather
mild assumption (see \cite{Ay}, Proposition 2.1.70, for criteria that this
holds) that
\begin{equation}\label{our generalized t struc}
(D^{\leq 0}, D^{\geq 0})
\end{equation} is a $t$-structure. This $t$-structure is standard in the same sense as before;
note that it is now a harder problem than in the abelian case to
describe all possible standard $t$-structures on $\KC$ (compare
Section \ref{other standards}). Then we define a \emph{generalized
AR-sequence} to be a simple object in the abelian category $\A :=
D^{\leq 0} \cap D^{\geq 0}$.

Note that if $\Cab$ is abelian to start with a generalized
AR-sequence will simply be a usual AR-sequence. We haven't worked
out the details, but generalized AR-sequences will certainly be
closely related to higher AR-sequences, whose definition we for
the sake of completeness recall here:

A \emph{higher AR-sequence}, see \cite{I}, in a suitable additive
category $\Cab$, is a long exact sequence
$$
0 \to X^{-n} \overset{d^{-n}}{\to} X^{-n+1}
\overset{d^{-n+1}}{\to} \ldots \to X^0 \to 0
$$
such that each $d^{-i}$ belongs to the radical of $\Cab$, $X^{-n}$
and $X^0$ are indecomposable and the sequence
$$
0 \to \Hom_{\Cab}(A,X^{-n}) \to \ldots \to \Hom_{\Cab}(A,X^{1})
\to J_{A,X^0} \to 0
$$
is exact for all $A \in \Ob(\Cab)$, where $J_{A,X^0}$ is the
radical of  $\Hom_{\Cab}(A,X^{0})$. (When $n=2$ this gives usual
AR-sequences.)
\subsection{Motivation from representation theory: Category
$\BGG$.}\label{Soergel section} Let $\g$ be  a semi-simple complex
Lie algebra and let $R$ be the cohomology ring of the flag
manifold of $\g$. Let $\BGG_0$ be the principal block of the
Bernstein-Gelfand-Gelfand category $\BGG$ of representations of
$\g$. Then by Soergel's theory, \cite{S}, $\Proj(\BGG_0)$ is
equivalent to an additive Karoubi closed subcategory $\Cab$ of
$R$-mod. ($\Cab$ is not abelian unless $\g = \mathfrak{sl}_2$.)
Thus,
$$\operatorname{D}^b(\BGG_0) \cong \Kb(\Proj(\BGG_0)) \cong \KC.$$
Hence the tautological $t$-structure on $\operatorname{D}^b(\BGG_0)$
corresponds to the $t$-structure on $\KC$ which is given by
\ref{our generalized t struc}.

If $\g = \mathfrak{sl}_2$, then $R = \C[x]/(x^2)$, $\Cab = \Rmod$
and the standard $t$-structure on $\operatorname{D}^b(\BGG_0)$
corresponds to the $t$-structure from Proposition \ref{p1} on
$\KC$; hence the category $\A$ is equivalent to $\BGG_0$ in this
case. The usual duality on $\BGG$ is also obtained by a general
construction that we give in the next section.

In fact, this procedure may be generalized as follows. Start, say,
with an additive Karoubi closed subcategory $\Cab$ of the category
of all modules over a Frobenius algebra $R$ and consider the heart
$\A$ of a $t$-structure on $\KC$ of the form \ref{our generalized
t struc}. One may then ask interesting questions such as: If we
assume that $R$ is the cohomology ring $H^*(X)$, for some compact
manifold $X$, when can the heart, like category $\BGG$, then be
realized as a category of perverse sheaves on $X$? When is the
heart Koszul, etc? (Compare with \cite{BGS}.)

\subsection{Duality over a Frobenius algebra}\label{final remarks}
Let $R$ be a commutative Frobenius algebra. Then the classes of
injective and projective $R$-modules coincide and the duality
functor
$$\Rmod \to \Rmod, \ M \mapsto M^* := \Hom_{\Rmod}(M,R)$$ fixes the
projective modules. In this case we can define a duality functor
${}^*$ on $\A = \A(\Rmod)$ as follows.

First we define the dual $(P_C)^*$ of $P_C$ for $C \in \Rmod$.
Pick an injective resolution $0 \to C \to I \overset{d}{\to} I'$
and define
$$
(\Pp_C)^* = [\Ker d^* \to I'^* \overset{d^*}{\to} I^*]
$$
$(\Pp_C)^*$ is a well-defined object in $\A$ since different
injective resolutions of the same object are homotopic. Next, for
a general $\A$-object $X = [A \to B \overset{f}{\to} C]$ we define
$$
X^* = \Ker ((\Pp_C)^* \overset{f'}{\to} (\Pp_B)^*)
$$
where the map $f'$ is naturally induced by $f$. Then some diagram
chasing proves that $X \mapsto X^*$ gives a well defined
contravariant functor ${}^*:\A \to \A$ whose square is equivalent
to the identity.
\begin{example}\label{2 the coinvariants} In the notations of example \ref{the coinvariants} we
have $P^*_{\FF} \cong [\FF \to R \to R]$, while  $P_{\FF}$, $P_R$
and $[0 \to \FF \to R]$ are selfdual.
\end{example}

\end{document}